\documentclass[reqno]{amsart}
\usepackage{amssymb}

\newtheorem{theorem}{Theorem}[section]

\newtheorem{lemma}[theorem]{Lemma}

\theoremstyle{definition}
\newtheorem{definition}[theorem]{Definition}

\newtheorem{notation}[theorem]{Notation}
\allowdisplaybreaks
\numberwithin{equation}{section}
\begin{document}

\title[Diagonal hypersurfaces and hypergeometric function]
{Number of $\mathbb{F}_q$-points on Diagonal hypersurfaces and hypergeometric function}


 \author{Sulakashna}
\address{Department of Mathematics, Indian Institute of Technology Guwahati, North Guwahati, Guwahati-781039, Assam, INDIA}
\curraddr{}
\email{sulakash@iitg.ac.in}
\author{Rupam Barman}
\address{Department of Mathematics, Indian Institute of Technology Guwahati, North Guwahati, Guwahati-781039, Assam, INDIA}
\curraddr{}
\email{rupam@iitg.ac.in}

\thanks{}


\date{20th October 2022, version-1}
\begin{abstract} 
Let $D_\lambda^d$ denote the family of monomial deformations of diagonal hypersurface over a finite field $\mathbb{F}_q$ given by
\begin{align*}
	D_\lambda^d: X_1^d+X_2^d+\cdots+X_n^d=\lambda d X_1^{h_1}X_2^{h_2}\cdots X_n^{h_n},
\end{align*} 
where $d,n\geq2$, $h_i\geq1$, $\sum_{i=1}^n h_i=d$, and $\gcd(d,h_1,h_2,\ldots, h_n)=1$.  The Dwork hypersurface is the case when $d=n$, that is, $h_1=h_2=\cdots =h_n=1$. 
Formulas for the number of  $\mathbb{F}_q$-points on the Dwork hypersurfaces in terms of McCarthy's $p$-adic hypergeometric functions are known. 
In this article we provide a formula for the number of $\mathbb{F}_q$-points on $D_\lambda^d$ in terms of McCarthy's $p$-adic hypergeometric function which holds for $d\geq n$. 
\end{abstract}
\maketitle
\section{Introduction and statement of results}
For an odd prime $p$, let $\mathbb{F}_q$ denote the finite field with $q$ elements, where $q=p^r, r\geq 1$. 
Let $D_\lambda^d$ denote the family of \emph{monomial deformations of diagonal hypersurface} over a finite field $\mathbb{F}_q$. These families are of the form:
\begin{align}\label{diagonal}
D_\lambda^d: X_1^d+X_2^d+\cdots+X_n^d=\lambda d X_1^{h_1}X_2^{h_2}\cdots X_n^{h_n},
\end{align} 
where $d,n\geq2$, $h_i\geq1$, $\sum_{i=1}^n h_i=d$, and $\gcd(d,h_1,h_2,\ldots, h_n)=1$. Let $\#D_{\lambda}^d(\mathbb{F}_q)$ denote the number of points on the hypersurface \eqref{diagonal} in $\mathbb{P}^{n-1}(\mathbb{F}_q)$. 
The Dwork hypersurface is the case when $d=n$, that is, $h_1=h_2=\cdots =h_n=1$. 
\par Let $d$ be an odd integer, and let $q\equiv 1\pmod{d}$. In \cite{goodson1}, Goodson expressed the number of points on the Dwork hypersurface over $\mathbb{F}_q$ as $(q^{d-1}-1)/(q-1)$ 
plus a sum of finite field hypergeometric functions as defined by Greene \cite{greene}. Furthermore, she conjectured that the same is true for $d$ even (see for example \cite[Conjecture 1.3]{goodson1}) 
which is open to date. On the other hand, formulas for the number of points on the Dwork hypersurface over a finite field $\mathbb{F}_q$ in terms of McCarthy's $p$-adic hypergeometric function $G[\cdots]_q$ (see Definition \ref{defin1} 
below) are known in the general case (see for example \cite{BRS, goodson, mccarthy3}).  
\par The aim of this paper is to find a formula for the number of $\mathbb{F}_q$-points on the hypersurface \eqref{diagonal} which holds for $d>n$ as well (the non Dwork case). In \cite{kob2}, Koblitz expressed $\#D_{\lambda}^d(\mathbb{F}_q)$ 
in terms of Gauss and Jacobi sums under the condition that $q\equiv 1\pmod{d}$. Using Koblitz's formula, Salerno \cite{salerno} expressed $\#D_{\lambda}^d(\mathbb{F}_q)$ as a sum of hypergeometric functions over finite fields as defined by Katz \cite{katz}, 
under the condition that $dh_1\cdots h_n \mid q-1$. This problem has not been studied to date when $q\not \equiv 1\pmod{d}$ and $d>n$. In this article we find a formula for $\#D_{\lambda}^d(\mathbb{F}_q)$ in terms of McCarthy's $p$-adic hypergeometric function when $\gcd(d, q-1)=1$. Before we state our main results, we define rational numbers $b_1, b_2, \ldots, b_{d-1}$ as follows. 
\begin{notation}\label{notation} We have two cases according to $d=n$ or $d>n$.
\begin{itemize}
\item ($d=n$) In this case $h_1=\cdots=h_n=1$. We take $b_1=b_2= \cdots =b_{d-1}=0$. 
\item $(d>n)$ In this case one or more $h_i$ is greater than $1$. We take $b_1= b_2= \cdots =b_{n-1}=0$; and corresponding to each $h_j>1$, we take $b_{j_1}=\frac{1}{h_j}, b_{j_2}=\frac{2}{h_j}, \ldots, b_{j_{h_j-1}}=\frac{h_j-1}{h_j}$ 
contributing to the remaining $d-n$ values of $b_i$ $(i=n, n+1, \ldots, d-1)$. For example, if $h_1=h_3=1, h_2=2, h_4=3$, then $n=4$, $d=h_1+h_2+h_3+h_4=7$, and hence $b_1= b_2=b_3=0$, $b_4=\frac{1}{2}$ (corresponding to $h_2$), and $b_5=\frac{1}{3}, b_6=\frac{2}{3}$ (corresponding to $h_4$).
\end{itemize}
\end{notation}
We now state our first main theorem.
\begin{theorem}\label{MT-1}
Let $p$ be an odd prime and $q=p^r, \ r\geq 1$. Let $d\geq2$ be an integer such that $\gcd(d,q-1)=1$. If $D_\lambda^d$ is the diagonal hypersurface given by \eqref{diagonal} such that $\lambda \neq 0$ and $p\nmid dh_1\cdots h_n$, 
then the number of points on $D_\lambda^d$ in $\mathbb{P}^{n-1}(\mathbb{F}_q)$ is given by
\begin{align*}
\#D_\lambda^d(\mathbb{F}_q) = \frac{q^{n-1}-1}{q-1} +(-1)^{n}\cdot {_{d-1}}G_{d-1}\left[\begin{array}{cccc}
	\frac{1}{d}, & \hspace{-.2cm} \frac{2}{d}, & \hspace{-.2cm} \ldots, & \hspace{-.2cm} \frac{d-1}{d}\vspace{.1cm} \\
	b_1, & \hspace{-.2cm} b_2, &\hspace{-.2cm}  \ldots, & \hspace{-.2cm} b_{d-1}
	\end{array}|\lambda^d h_1^{h_1}\cdots h_n^{h_n}
	\right]_q,
\end{align*}
where the $b_i$'s are as given in Notation \ref{notation}.
\end{theorem}
If $d$ is a prime, then the condition $\gcd(d,q-1)=1$ is equivalent to $q\not \equiv 1\pmod{d}$. The condition $\gcd(d,q-1)=1$ is very crucial in the proof of Theorem \ref{MT-1}. Finding a formula for $\#D_{\lambda}^d(\mathbb{F}_q)$ in terms of 
McCarthy's $p$-adic hypergeometric function without assuming the condition $\gcd(d, q-1)=1$ seems to be a difficult problem. In the next theorem, we find such a formula without assuming the condition $\gcd(d, q-1)=1$ for a particular family of diagonal hypersurfaces, namely
\begin{align}\label{eq-4.0}
D_\lambda^{d,k}:x_1^d+x_2^d=d\lambda x_1^k x_2^{d-k},
\end{align}
where $d\geq 2$, $k\geq 1$, and $\gcd(d,k,d-k)=1$. The second author with Saikia \cite{BS5} expressed the number of points on \eqref{eq-4.0} over $\mathbb{F}_p$ in terms of McCarthy's $p$-adic hypergeometric function when $k=1$. 
In the following theorem, we express the number of points on the hypersurface \eqref{eq-4.0} over $\mathbb{F}_q$ for any $k\geq 1$ and without assuming the condition $\gcd(d,q-1)=1$.
\begin{theorem}\label{MT-2}
	Let $p$ be an odd prime and $q=p^r, \ r\geq 1$. Let $d\geq2$ and $k\geq 1$ be integers, and let $D_\lambda^{d,k}$ be the diagonal hypersurface given in \eqref{eq-4.0} such that $p\nmid dk(d-k)$. 
	Then, for $\lambda \neq 0$, the number of points on $D_\lambda^{d,k}$ in $\mathbb{P}^{1}(\mathbb{F}_q)$ is given by
	\begin{align*}
	\#D_\lambda^{d,k}(\mathbb{F}_q) = 1 + {_{d-1}}G_{d-1}\left[\begin{array}{cccc}
	\frac{1}{d}, &\hspace{-.2cm} \frac{2}{d}, &\hspace{-.2cm} \ldots, &\hspace{-.2cm} \frac{d-1}{d}\vspace{.1cm} \\
	b_1, &\hspace{-.2cm} b_2, &\hspace{-.2cm}  \ldots, &\hspace{-.2cm} b_{d-1}
	\end{array}|\lambda^d k^{k} (d-k)^{d-k}
	\right]_q,
	\end{align*}
	where $b_i$'s are as given in Notation \ref{notation} with $h_1=k$ and $h_2=d-k$.
\end{theorem}
\section{Preliminaries}
For an odd prime $p$, let $\mathbb{F}_q$ denote the finite field with $q$ elements, where $q=p^r, r\geq 1$. Let $\widehat{\mathbb{F}_q^{\times}}$ be the group of all multiplicative
characters on $\mathbb{F}_q^{\times}$. We extend the domain of each $\chi\in \widehat{\mathbb{F}_q^{\times}}$ to $\mathbb{F}_q$ by setting $\chi(0):=0$
including the trivial character $\varepsilon$. We denote by $\overline{\chi}$ the character inverse of $\chi$. The orthogonality relation satisfied by multiplicative characters is given as follows:
\begin{align}\label{eq-1}
	\sum\limits_{x\in {\mathbb{F}_q }} \chi(x)= \left\{
\begin{array}{ll}
q-1, & \hbox{if $\chi=\varepsilon$;} \\
0, & \hbox{otherwise.}
\end{array}
\right.
\end{align}
Let $\mathbb{Z}_p$ and $\mathbb{Q}_p$ denote the ring of $p$-adic integers and the field of $p$-adic numbers, respectively.
Let $\overline{\mathbb{Q}_p}$ be the algebraic closure of $\mathbb{Q}_p$ and $\mathbb{C}_p$ the completion of $\overline{\mathbb{Q}_p}$.
Let $\mathbb{Z}_q$ be the ring of integers in the unique unramified extension of $\mathbb{Q}_p$ with residue field $\mathbb{F}_q$.
We know that $\chi\in \widehat{\mathbb{F}_q^{\times}}$ takes values in $\mu_{q-1}$, where $\mu_{q-1}$ is the group of
$(q-1)$-th roots of unity in $\mathbb{C}^{\times}$. Since $\mathbb{Z}_q^{\times}$ contains all $(q-1)$-th roots of unity,
we can consider multiplicative characters on $\mathbb{F}_q^\times$
to be maps $\chi: \mathbb{F}_q^{\times} \rightarrow \mathbb{Z}_q^{\times}$.
Let $\omega: \mathbb{F}_q^\times \rightarrow \mathbb{Z}_q^{\times}$ be the Teichm\"{u}ller character.
For $a\in\mathbb{F}_q^\times$, the value $\omega(a)$ is just the $(q-1)$-th root of unity in $\mathbb{Z}_q$ such that $\omega(a)\equiv a \pmod{p}$.
\par Next, we introduce the Gauss sum and recall some of its elementary properties. For further details, see \cite{evans}. Let $\zeta_p$ be a fixed primitive $p$-th root of unity
in $\overline{\mathbb{Q}_p}$. The trace map $\text{tr}: \mathbb{F}_q \rightarrow \mathbb{F}_p$ is given by
\begin{align}
\text{tr}(\alpha)=\alpha + \alpha^p + \alpha^{p^2}+ \cdots + \alpha^{p^{r-1}}.\notag
\end{align}
Then the additive character
$\theta: \mathbb{F}_q \rightarrow \mathbb{Q}_p(\zeta_p)$ is defined by
\begin{align}
\theta(\alpha)=\zeta_p^{\text{tr}(\alpha)}.\notag
\end{align}
For $\chi \in \widehat{\mathbb{F}_q^\times}$, the \emph{Gauss sum} is defined by
\begin{align}
g(\chi):=\sum\limits_{x\in \mathbb{F}_q}\chi(x)\theta(x) .\notag
\end{align}
We now fix a generator of $\widehat{\mathbb{F}_q^{\times}}$ which we denote by $T$.
\begin{lemma}\emph{(\cite[Lemma 2.2]{Fuselier}).}\label{lemma2_2} For $\alpha\in \mathbb{F}_q$, we have
	\begin{align}
		\theta(\alpha)=\frac{1}{q-1}\sum_{m=0}^{q-2} g(T^{-m})T^m(\alpha).\notag
	\end{align}
\end{lemma}
Now, we recall the $p$-adic gamma function. For further details, see \cite{kob}.
For a positive integer $n$,
the $p$-adic gamma function $\Gamma_p(n)$ is defined as
\begin{align}
\Gamma_p(n):=(-1)^n\prod\limits_{0<j<n,p\nmid j}j\notag
\end{align}
and one extends it to all $x\in\mathbb{Z}_p$ by setting $\Gamma_p(0):=1$ and
\begin{align}
\Gamma_p(x):=\lim_{x_n\rightarrow x}\Gamma_p(x_n)\notag
\end{align}
for $x\neq0$, where $x_n$ runs through any sequence of positive integers $p$-adically approaching $x$.
This limit exists, is independent of how $x_n$ approaches $x$,
and determines a continuous function on $\mathbb{Z}_p$ with values in $\mathbb{Z}_p^{\times}$.
Let $\pi \in \mathbb{C}_p$ be the fixed root of $x^{p-1} + p=0$ which satisfies
$\pi \equiv \zeta_p-1 \pmod{(\zeta_p-1)^2}$. Then the Gross-Koblitz formula relates Gauss sums and the $p$-adic gamma function as follows.
\begin{theorem}\emph{(\cite[Gross-Koblitz]{gross}).}\label{thm2_3} For $a\in \mathbb{Z}$ and $q=p^r, r\geq 1$, we have
\begin{align}
g(\overline{\omega}^a)=-\pi^{(p-1)\sum\limits_{i=0}^{r-1}\langle\frac{ap^i}{q-1} \rangle}\prod\limits_{i=0}^{r-1}\Gamma_p\left(\left\langle \frac{ap^i}{q-1} \right\rangle\right).\notag
\end{align}
\end{theorem}
Now, we recall McCarthy's $p$-adic hypergeometric function. For $x \in \mathbb{Q}$, we let $\lfloor x\rfloor$ denote the greatest integer less than or equal to $x$ and $\langle x\rangle$ 
denote the fractional part of $x$, that is, $x-\lfloor x\rfloor$, satisfying $0\leq\langle x\rangle<1$. McCarthy's $p$-adic hypergeometric function $_{n}G_{n}[\cdots]_q$ is defined as follows.
\begin{definition}\cite[Definition 5.1]{mccarthy2} \label{defin1}
	Let $p$ be an odd prime and $q=p^r$, $r\geq 1$. Let $t \in \mathbb{F}_q$.
	For positive integer $n$ and $1\leq k\leq n$, let $a_k$, $b_k$ $\in \mathbb{Q}\cap \mathbb{Z}_p$.
	Then the function $_{n}G_{n}[\cdots]_q$ is defined by
	\begin{align}
		&_nG_n\left[\begin{array}{cccc}
			a_1, & a_2, & \ldots, & a_n \\
			b_1, & b_2, & \ldots, & b_n
		\end{array}|t
		\right]_q\notag\\
		&\hspace{1cm}:=\frac{-1}{q-1}\sum_{a=0}^{q-2}(-1)^{an}~~\overline{\omega}^a(t)
		\prod\limits_{k=1}^n\prod\limits_{i=0}^{r-1}(-p)^{-\lfloor \langle a_kp^i \rangle-\frac{ap^i}{q-1} \rfloor -\lfloor\langle -b_kp^i \rangle +\frac{ap^i}{q-1}\rfloor}\notag\\
		&\hspace{2cm} \times \frac{\Gamma_p(\langle (a_k-\frac{a}{q-1})p^i\rangle)}{\Gamma_p(\langle a_kp^i \rangle)}
		\frac{\Gamma_p(\langle (-b_k+\frac{a}{q-1})p^i \rangle)}{\Gamma_p(\langle -b_kp^i \rangle)}.\notag
	\end{align}
\end{definition}
In the proofs of our main results, we need to consider certain product formulas for the $p$-adic gamma function. The following lemmas relate certain products of values of the $p$-adic gamma function.
\begin{lemma}\emph{(\cite[Lemma 3.1]{BS1})}\label{lemma3_1}
Let $p$ be a prime and $q=p^r, r\geq 1$. For $0\leq a\leq q-2$ and $t\geq 1$ with $p\nmid t$, we have
\begin{align}
\omega(t^{-ta})\prod\limits_{i=0}^{r-1}\Gamma_p\left(\hspace{-.1cm}\left\langle\frac{-tp^ia}{q-1}\right\rangle
\hspace{-.1cm}\right)
\prod\limits_{h=1}^{t-1}\Gamma_p\left(\hspace{-.1cm}\left\langle \frac{hp^i}{t}\right\rangle\hspace{-.1cm}\right)
=\prod\limits_{i=0}^{r-1}\prod\limits_{h=0}^{t-1}\Gamma_p\left(\hspace{-.1cm}\left\langle\frac{p^i(1+h)}{t}-\frac{p^ia}{q-1}\right\rangle \hspace{-.1cm}\right).\notag
\end{align}
\end{lemma}
\begin{lemma}\emph{(\cite[Lemma 3.1]{BS1})}\label{lemma3_2}
Let $p$ be a prime and $q=p^r, r\geq 1$. For $0\leq a\leq q-2$ and $t\geq 1$ with $p\nmid t$, we have
\begin{align*}
\omega(t^{ta})\prod\limits_{i=0}^{r-1}\Gamma_p\left(\hspace{-.1cm} \left\langle\frac{tp^ia}{q-1}\right\rangle\hspace{-.1cm}\right)
\prod\limits_{h=1}^{t-1}\Gamma_p\left(\hspace{-.1cm}\left\langle \frac{hp^i}{t}\right\rangle\hspace{-.1cm}\right)
=\prod\limits_{i=0}^{r-1}\prod\limits_{h=0}^{t-1}\Gamma_p\left(\hspace{-.1cm} \left\langle\frac{p^i h}{t}+\frac{p^ia}{q-1}\right\rangle \hspace{-.1cm}\right).\notag
\end{align*}
\end{lemma}
We prove the following lemmas which will be used to prove our main results.
\begin{lemma}\label{lemma-3.2}
Let $p$ be an odd prime and $q=p^r, r\geq 1$. Let $d\geq2$ be an integer such that $p\nmid d$. Then, for $1\leq a\leq q-2$ and $0\leq i\leq r-1$, we have
\begin{align}\label{eq-5}
\left\lfloor\frac{ap^i}{q-1}\right\rfloor +\left\lfloor\frac{-dap^i}{q-1}\right\rfloor = \sum_{h=1}^{d-1} \left\lfloor\left\langle\frac{hp^i}{d}\right\rangle-\frac{ap^i}{q-1}\right\rfloor -1.
\end{align}
\end{lemma}
\begin{proof}
Let $\left\lfloor\frac{-dap^i}{q-1}\right\rfloor=dk+s$, where $k, s \in \mathbb{Z}$ and $0\leq s\leq d-1$.
This yields
\begin{align*}
dk+s \leq \frac{-dap^i}{q-1}< dk+s+1,
\end{align*}
which implies
\begin{align}\label{eq-6}
k+\frac{s}{d} \leq \frac{-ap^i}{q-1}< k+\frac{s}{d}+\frac{1}{d}
\end{align}
and
\begin{align}\label{eq-7}
-k-\frac{s}{d}-\frac{1}{d} < \frac{ap^i}{q-1} \leq -k-\frac{s}{d}.
\end{align}
Since $a\neq 0$, by \eqref{eq-7} we obtain $\left\lfloor\frac{ap^i}{q-1}\right\rfloor = -k-1$ and hence the left hand side of \eqref{eq-5} becomes $(d-1)k+s-1$.
Since $p\nmid d$, we observe that 
\begin{align*}
\sum_{h=1}^{d-1} \left\lfloor\left\langle\frac{hp^i}{d}\right\rangle-\frac{ap^i}{q-1}\right\rfloor = \sum_{h=1}^{d-1} \left\lfloor\left\langle\frac{h}{d}\right\rangle-\frac{ap^i}{q-1}\right\rfloor.
\end{align*}
For $1\leq h\leq d-1$, \eqref{eq-6} gives 
\begin{align}\label{eq-7.0}
k+\frac{s}{d}+\frac{h}{d} \leq \frac{h}{d} - \frac{ap^i}{q-1}< k+\frac{s}{d}+\frac{1}{d}+\frac{h}{d}.
\end{align}
If $s=0$, then $\left\lfloor\left\langle\frac{h}{d}\right\rangle-\frac{ap^i}{q-1}\right\rfloor=k$ for $h=1,\ldots,d-1$ and hence 
\begin{align*}
	\sum_{h=1}^{d-1} \left\lfloor\left\langle\frac{h}{d}\right\rangle-\frac{ap^i}{q-1}\right\rfloor =(d-1)k. 
\end{align*}
If $1\leq s \leq d-1$, then \eqref{eq-7.0} yields
\begin{align}\label{eq-7.1}
\left\lfloor\left\langle\frac{h}{d}\right\rangle-\frac{ap^i}{q-1}\right\rfloor=\left\{
                                              \begin{array}{ll}
                                                k, & \hbox{if $1\leq h\leq d-s-1$;} \\
                                                k+1, & \hbox{if $d-s \leq h \leq d-1$.}
                                              \end{array}
                                            \right.
\end{align}
From \eqref{eq-7.1}, we obtain 
\begin{align*}
	\sum_{h=1}^{d-1} \left\lfloor\left\langle\frac{h}{d}\right\rangle-\frac{ap^i}{q-1}\right\rfloor = (d-1)k+s.
\end{align*}
 Thus, for $0\leq s\leq d-1$, we have
 \begin{align*}
 	\sum_{h=1}^{d-1} \left\lfloor\left\langle\frac{h}{d}\right\rangle-\frac{ap^i}{q-1}\right\rfloor -1 = (d-1)k+s-1.
 \end{align*}
 Therefore, the right hand side of \eqref{eq-5} becomes $(d-1)k+s-1$. This completes the proof of the lemma. 
\end{proof}
\begin{lemma}\label{lemma-3_3}
Let $p$ be an odd prime and $q=p^r, r\geq 1$. Let $l$ be a positive integer such that $p\nmid l$. Then, for $0\leq a\leq q-2$ and $0\leq i\leq r-1$, we have
\begin{align}\label{eq-10}
\left\lfloor\frac{lap^i}{q-1}\right\rfloor = \sum_{h=0}^{l-1} \left\lfloor\left\langle\frac{-hp^i}{l}\right\rangle+\frac{ap^i}{q-1}\right\rfloor.
\end{align}
\end{lemma}
\begin{proof}
Let $\left\lfloor\frac{lap^i}{q-1}\right\rfloor=lk+s$, where $k, s \in \mathbb{Z}$ and $0\leq s\leq l-1$.
Then
\begin{align*}
lk+s \leq \frac{lap^i}{q-1}< lk+s+1,
\end{align*}
which implies
\begin{align}\label{eq-11}
k+\frac{s}{l} \leq \frac{ap^i}{q-1}< k+\frac{s}{l}+\frac{1}{l}.
\end{align}
Using \eqref{eq-11}, we have $\left\lfloor\frac{ap^i}{q-1}\right\rfloor = k$. 
Since $p\nmid l$, we observe that 
\begin{align*}
\sum_{h=1}^{l-1} \left\lfloor\left\langle\frac{-hp^i}{l}\right\rangle+\frac{ap^i}{q-1}\right\rfloor = \sum_{h=1}^{l-1} \left\lfloor\left\langle\frac{-h}{l}\right\rangle+\frac{ap^i}{q-1}\right\rfloor.
\end{align*}
For $1\leq h\leq l-1$, \eqref{eq-11} gives 
\begin{align}\label{eq-11.0}
k+\frac{s}{l}-\frac{h}{l}+1 \leq \left\langle\frac{-h}{l}\right\rangle + \frac{ap^i}{q-1}< k+\frac{s}{l}+\frac{1}{l}-\frac{h}{l}+1.
\end{align}
If $s=0$, then $\left\lfloor\left\langle\frac{-h}{l}\right\rangle+\frac{ap^i}{q-1}\right\rfloor=k$ for all $1\leq h\leq l-1$ and hence 
\begin{align*}
\sum_{h=1}^{l-1} \left\lfloor\left\langle\frac{-h}{l}\right\rangle+\frac{ap^i}{q-1}\right\rfloor =(l-1)k. 
\end{align*}
If $1\leq s \leq l-1$, then \eqref{eq-11.0} yields
\begin{align}\label{eq-11.1}
\left\lfloor\left\langle\frac{-h}{l}\right\rangle+\frac{ap^i}{q-1}\right\rfloor=\left\{
                                              \begin{array}{ll}
                                                k+1, & \hbox{if $1\leq h\leq s$;} \\
                                                k, & \hbox{if $s+1 \leq h \leq l-1$.}
                                              \end{array}
                                            \right.
\end{align}
From \eqref{eq-11.1}, we obtain 
\begin{align*}
	\sum_{h=1}^{l-1} \left\lfloor\left\langle\frac{-h}{l}\right\rangle+\frac{ap^i}{q-1}\right\rfloor = (l-1)k+s.
\end{align*}
Thus, for $0\leq s\leq l-1$, we have
 \begin{align*}
 	\sum_{h=0}^{l-1} \left\lfloor\left\langle\frac{-h}{l}\right\rangle+\frac{ap^i}{q-1}\right\rfloor = k+(l-1)k+s=lk+s.
 \end{align*}
This completes the proof of the lemma.
\end{proof}
\begin{lemma}\label{lemma-3.4}
Let $p$ be an odd prime and $q=p^r, r\geq 1$. Let $d\geq 2$ be an integer such that $\gcd(d,p(q-1))=1$. Let $h_1,\ldots,h_n$ be positive integers such that $\sum_{k=1}^n h_k=d$ and $\gcd(h_1\cdots h_n, p)=1$ Then, for $1\leq a\leq q-2$ and $0\leq i\leq r-1$, we have
\begin{align*}
&\left\lfloor\frac{-dap^i}{q-1}\right\rfloor +\sum_{k=1}^{n} \left\lfloor\frac{h_k ap^i}{q-1}\right\rfloor\nonumber\\
&= \sum_{h=1}^{d-1} \left\lfloor\left\langle\frac{hp^i}{d}\right\rangle -\frac{ap^i}{q-1}\right\rfloor+ \sum_{k=1}^{n} \sum_{h=0}^{h_k-1} \left\lfloor\left\langle\frac{-hp^i}{h_k}\right\rangle+\frac{ap^i}{q-1}\right\rfloor -1-\left\lfloor \frac{ap^i}{q-1}\right\rfloor.
\end{align*}
\end{lemma}
\begin{proof}
Putting  $l=h_k$ for $k=1, 2, \ldots, n$ in \eqref{eq-10}, and then adding them we have
\begin{align}\label{new-eq-1}
\sum_{k=1}^{n} \left\lfloor\frac{h_k ap^i}{q-1}\right\rfloor= \sum_{k=1}^{n} \sum_{h=0}^{h_k-1} \left\lfloor\left\langle\frac{-hp^i}{h_k}\right\rangle+\frac{ap^i}{q-1}\right\rfloor.
\end{align}
We now readily obtain the required identity by adding \eqref{eq-5} and \eqref{new-eq-1}.
\end{proof}
Finally, we recall an identity from \cite{BS1}.
\begin{lemma}\emph{(\cite[Eq. (23)]{BS1})}\label{lemma-3.3}
Let $p$ be an odd prime and $q=p^{r}, r\geq 1$. For $0<a\leq q-2$, we have
\begin{align*}
\prod_{i=0}^{r-1} \Gamma_{p}\left(\left\langle\left(1-\frac{a}{q-1}\right)p^{i}\right\rangle\right)\Gamma_{p}\left(\left\langle\frac{ap^{i}}{q-1}\right\rangle\right) = (-1)^r \overline{\omega}^{a}(-1).
\end{align*}
\end{lemma}
 \section{Proof of Theorem \ref{MT-1} and Theorem \ref{MT-2}}
 \begin{proof}[Proof of Theorem \ref{MT-1}]
 Let $N_q^d(\lambda)$ denote the number of points on the diagonal hypersurface $D_\lambda^d$ in $\mathbb{A}^{n}(\mathbb{F}_q)$. Then we have
\begin{align}\label{eq-12.1}
\# D_\lambda^d (\mathbb{F}_q)=\frac{N_q^d(\lambda) -1}{q-1}. 
\end{align}
Let $\overline{x}=(x_1,x_2,\ldots, x_n)$ and $f(\overline{x})=x_1^d+x_2^d+\cdots+x_n^d-d\lambda x_1^{h_1}\cdots x_n^{h_n}$ and using the identity 
\begin{align*}
\sum_{z\in \mathbb{F}_q} \theta(zf(\overline{x}))= \left\{
\begin{array}{ll}
q, & \hbox{if $f(\overline{x})=0$;} \\
0, & \hbox{if $f(\overline{x})\neq 0$,}
\end{array}
\right.
\end{align*}
we obtain
\begin{align}\label{eq-13}
q\cdot N_q^d(\lambda)&= \sum_{z\in\mathbb{F}_q}\sum_{x_i \in\mathbb{F}_q}\theta(zf(\overline{x}))\nonumber\\
&=q^n + \sum_{z,x_i \in\mathbb{F}_q^{\times}}\theta(zf(\overline{x})) +\sum_{z\in\mathbb{F}_q^{\times}}\sum_{\begin{subarray}{1} \ \text{some } \\ \  x_i =0\end{subarray}} \theta(zf(\overline{x})).
\end{align}
Consider a polynomial $f_1(\overline{x})=x_1^d+\cdots + x_n^d$ and let $N_q^{'}$ be the number of solutions of $f_1(\overline{x})$ in $\mathbb{A}^n(\mathbb{F}_q)$.
Since $d$ is an integer such that $\gcd(d, q-1)=1$, therefore $x \mapsto x^d$ is an automorphism of $\mathbb{F}_q^\times$. This gives that $N_q^{'}=q^{n-1}$.
Also, repeating the same process for $f_1(\overline{x})$ as done in \eqref{eq-13} for $f(\overline{x})$, we deduce that 
\begin{align*}
q\cdot N_q^{'}=q^n+  \sum_{z,x_i \in\mathbb{F}_q^{\times}}\theta(zf_1(\overline{x})) +\sum_{z\in\mathbb{F}_q^{\times}}\sum_{\begin{subarray}{1} \ \text{some } \\ \  x_i =0\end{subarray}} \theta(zf_1(\overline{x})).
\end{align*}
Thus,
\begin{align}\label{eq-13.1}
  \sum_{z\in\mathbb{F}_q^{\times}}\sum_{\begin{subarray}{1} \ \text{some } \\ \  x_i =0\end{subarray}} \theta(zf_1(\overline{x}))=-\sum_{z,x_i \in\mathbb{F}_q^{\times}}\theta(zf_1(\overline{x})).
\end{align}
Also,
\begin{align}\label{eq-14}
\sum_{z\in\mathbb{F}_q^{\times}}\sum_{\begin{subarray}{1} \ \text{some } \\ \  x_i =0\end{subarray}} \theta(zf_1(\overline{x})) = \sum_{z\in\mathbb{F}_q^{\times}}\sum_{\begin{subarray}{1} \ \text{some } \\ \  x_i =0\end{subarray}} \theta(zf(\overline{x})).
\end{align}
Combining \eqref{eq-13}, \eqref{eq-13.1}, and \eqref{eq-14}, we obtain
\begin{align}\label{eq-17}
q\cdot N_q^d(\lambda)&=q^n + \sum_{z,x_i \in\mathbb{F}_q^{\times}}\theta(zf(\overline{x})) -\sum_{z,x_i \in\mathbb{F}_q^{\times}} \theta(zf_1(\overline{x}))\nonumber\\
&=q^n+A-B,
\end{align}
where $A:=\sum_{z,x_i \in\mathbb{F}_q^{\times}}\theta(zf(\overline{x}))$ and $B:=\sum_{z,x_i \in\mathbb{F}_q^{\times}} \theta(zf_1(\overline{x}))$. We now evaluate the values of $A$ and $B$, respectively. Firstly, we calculate $B$. We have
\begin{align*}
B&=\sum_{z,x_i \in\mathbb{F}_q^{\times}}\theta(z(x_1^d+\cdots+x_n^d))=\sum_{z,x_i \in\mathbb{F}_q^{\times}}\theta(z(x_1^d))\cdots\theta(z(x_n^d)).\\
\end{align*}
Lemma \ref{lemma2_2} yields
\begin{align*}
B&=\frac{1}{(q-1)^{n}}\sum_{z,x_i \in\mathbb{F}_q^{\times}} \sum_{a_1,\ldots,a_n=0}^{q-2}g(T^{-a_1})\cdots g(T^{-a_n})T^{a_1}(zx_1^d)\cdots T^{a_n}(zx_n^d)\\
&=\frac{1}{(q-1)^{n}} \sum_{a_1,\ldots,a_n=0}^{q-2} g(T^{-a_1})\cdots g(T^{-a_n})\\
&\hspace*{.5cm}\times\sum_{z\in\mathbb{F}_q^{\times}}T^{a_1+\cdots+a_n}(z)\sum_{x_1\in\mathbb{F}_q^{\times}} T^{a_1d}(x_1)\cdots\sum_{x_n\in\mathbb{F}_q^{\times}} T^{a_nd}(x_n).
\end{align*}
The inner sums are nonzero only if $a_1+\cdots +a_n\equiv 0\pmod{q-1}$ and $a_1d,\ldots, a_nd\equiv0\pmod{q-1}$. Since $\gcd(d, q-1)=1$, all the congruences simultaneously hold only if $a_1=\cdots= a_n =0$. Using the fact that $g(\varepsilon)=-1$, we obtain $B=(-1)^n (q-1).$\\
Next, we  calculate the value of $A$. We have
\begin{align*}
A&=\sum_{z,x_i \in\mathbb{F}_q^{\times}}\theta(zf(\overline{x}))=\sum_{z,x_i \in\mathbb{F}_q^{\times}}\theta(z(x_1^d+\cdots+x_n^d-\lambda d x_1^{h_1}\cdots x_n^{h_n}))\nonumber\\
&=\sum_{z,x_i \in\mathbb{F}_q^{\times}}\theta(zx_1^d)\cdots\theta(zx_n^d)\theta(-z\lambda d x_1^{h_1}\cdots x_n^{h_n}).
\end{align*}
By using Lemma \ref{lemma2_2}, we obtain
\begin{align*}
A=&\frac{1}{(q-1)^{n+1}} \sum_{a_1,\ldots,a_{n+1}=0}^{q-2} g(T^{-a_1})\cdots g(T^{-a_n})g(T^{-a_{n+1}})\\
&\hspace*{.5cm}\times\sum_{z,x_i\in\mathbb{F}_q^{\times}}T^{a_1}(zx_1^d)\cdots T^{a_n}(zx_n^d)T^{a_{n+1}}(-\lambda d zx_1^{h_1}\cdots x_n^{h_n})\\
&=\frac{1}{(q-1)^{n+1}} \sum_{a_1,\ldots,a_{n+1}=0}^{q-2}g(T^{-a_1})\cdots g(T^{-a_{n+1}})T^{a_{n+1}}(-\lambda d)\\
&\hspace*{.5cm}\times\sum_{z \in\mathbb{F}_q^{\times}}T^{a_1+\cdots+a_{n+1}}(z)\sum_{x_1\in\mathbb{F}_q^{\times}} T^{a_1d+a_{n+1}h_1}(x_1)\cdots\sum_{x_n\in\mathbb{F}_q^{\times}} T^{a_nd+a_{n+1}h_n}(x_n).
\end{align*}
The inner sums are nonzero only if $a_1+\cdots +a_{n+1}\equiv 0\pmod{q-1}$ and $a_1 d+a_{n+1}h_1,\ldots, a_nd+a_{n+1}h_n\equiv0\pmod{q-1}$. For $0\leq a\leq q-2$, we have $a_i\equiv h_i a\pmod{q-1}$  for $i=1,\ldots, n $ and $a_{n+1}\equiv -da\pmod{q-1}$ as $\gcd(d, q-1)=1$. Thus
\begin{align}\label{eq-16.1}
A&=\sum_{a=0}^{q-2}g(T^{-ah_1})\cdots g(T^{-ah_{n}})g(T^{ad})T^{-ad}(-\lambda d).
\end{align}
Taking $T=\omega$, and then using Gross-Koblitz formula, we obtain
\begin{align}\label{eqn-15}
A&=(-1)^{n+1}\sum_{a=0}^{q-2}\overline{\omega}^{ad}(-\lambda d)(-p)^{\displaystyle{\sum_{i=0}^{r-1}}\left( \displaystyle{\sum_{k=1}^{n}}\left\langle\frac{h_kap^i}{q-1}\right\rangle+\left\langle\frac{-dap^i}{q-1}\right\rangle\right)}\nonumber\\
&\hspace*{.5cm}\times\prod_{i=0}^{r-1}\Gamma_p\left(\left\langle\frac{-dp^ia}{q-1}\right\rangle          \right)\prod_{k=1}^{n}\Gamma_p\left(\left\langle\frac{h_k p^ia}{q-1}\right\rangle\right).
\end{align}
Applying Lemma \ref{lemma3_1} with $t=d$ and Lemma \ref{lemma3_2} with  $t=h_k$, we have
\begin{align*}
\omega(d^{-ad})\prod\limits_{i=0}^{r-1}\Gamma_p\left(\left\langle\frac{-dp^ia}{q-1}\right\rangle\right)
=\prod\limits_{i=0}^{r-1}\frac{\prod\limits_{h=0}^{d-1}\Gamma_p\left(\left\langle\frac{p^i(1+h)}{d}-\frac{p^ia}{q-1}\right\rangle \right)}{\prod\limits_{h=1}^{d-1}\Gamma_p\left(\left\langle \frac{hp^i}{d}\right\rangle\right)},\\
\omega(h_k^{h_ka})\prod\limits_{i=0}^{r-1}\Gamma_p\left(\left\langle\frac{h_kp^ia}{q-1}\right\rangle\right)
=\prod\limits_{i=0}^{r-1}\frac{\prod\limits_{h=0}^{h_k-1}\Gamma_p\left(\left\langle\frac{p^i h}{h_k
}+\frac{p^ia}{q-1}\right\rangle \right)}{\prod\limits_{h=1}^{h_k-1}\Gamma_p\left(\left\langle \frac{hp^i}{h_k}\right\rangle\right)}.
\end{align*}
Substituting these values in \eqref{eqn-15} yields
\begin{align*}
A&=(-1)^{n+1}\sum_{a=0}^{q-2}\overline{\omega}^{ad}(-\lambda d) \ \overline{\omega}( d^{-ad}) \ \overline{\omega}^{a}(h_1^{h_1}\cdots h_n^{h_n})(-p)^{\displaystyle{\sum_{i=0}^{r-1}}M_{a,i}}\\
&\hspace*{.5cm}\times\prod_{i=0}^{r-1}\prod\limits_{h=0}^{d-1}\frac{\Gamma_p\left(\left\langle\frac{p^i(1+h)}{d}-\frac{p^ia}{q-1}\right\rangle \right)}{\Gamma_p\left(\left\langle \frac{hp^i}{d}\right\rangle\right)}\prod_{k=1}^{n}
\left(\prod\limits_{h=0}^{h_k-1}\frac{\Gamma_p\left(\left\langle\frac{p^i h}{h_k
}+\frac{p^ia}{q-1}\right\rangle \right)}{ \Gamma_p\left(\left\langle \frac{hp^i}{h_k}\right\rangle\right)}\right),
\end{align*}
where $M_{a,i}=\displaystyle{\sum_{k=1}^{n}}\left\langle\frac{h_kap^i}{q-1}\right\rangle+\left\langle\frac{-dap^i}{q-1}\right\rangle=-\displaystyle{\sum_{k=1}^{n}}\left\lfloor\frac{h_kap^i}{q-1}\right\rfloor-\left\lfloor\frac{-dap^i}{q-1}\right\rfloor$.
We now have
\begin{align*}
A&=(-1)^{n+1} +(-1)^{n+1}\sum_{a=1}^{q-2}\overline{\omega}^{a}((-1)^d\lambda^d h_1^{h_1}\cdots h_n^{h_n})(-p)^{\sum_{i=0}^{r-1}M_{a,i}}\nonumber\\
&\hspace*{.5cm}\times\prod_{i=0}^{r-1}\Gamma_p\left(\left\langle p^i-\frac{p^ia}{q-1}\right\rangle \right)\Gamma_p\left(\left\langle\frac{p^ia}{q-1}\right\rangle \right)\Gamma_p^{n-1}\left(\left\langle\frac{p^ia}{q-1}\right\rangle \right)\nonumber\\
&\hspace*{.5cm}\times\prod\limits_{h=1}^{d-1}\frac{\Gamma_p\left(\left\langle\frac{p^ih}{d}-\frac{p^ia}{q-1}\right\rangle \right)}{\Gamma_p\left(\left\langle \frac{hp^i}{d}\right\rangle\right)}\prod\limits_{\begin{subarray}{1} \ k=1 \\ h_k>1 \end{subarray}}^{n}
\prod\limits_{h=1}^{h_k-1}\frac{\Gamma_p\left(\left\langle\frac{p^i h}{h_k
}+\frac{p^ia}{q-1}\right\rangle \right)}{\Gamma_p\left(\left\langle \frac{hp^i}{h_k}\right\rangle\right)}.
\end{align*}
Employing Lemma \ref{lemma-3.4} and Lemma \ref{lemma-3.3}, we obtain
\begin{align*}
A&=(-1)^{n+1} +(-1)^{n+1}q\sum_{a=1}^{q-2}\overline{\omega}^{a}((-1)^{d-1}\lambda^d h_1^{h_1}\cdots h_n^{h_n})(-p)^{\sum_{i=0}^{r-1}N_{a,i}}\\
&\hspace*{.5cm}\times\prod_{i=0}^{r-1}\Gamma_p^{n-1}\left(\left\langle\frac{p^ia}{q-1}\right\rangle \right)
\prod\limits_{h=1}^{d-1}\frac{\Gamma_p\left(\left\langle\frac{p^ih}{d}-\frac{p^ia}{q-1}\right\rangle \right)}{\Gamma_p\left(\left\langle \frac{hp^i}{d}\right\rangle\right)}\prod\limits_{\begin{subarray}{1} \ k=1 \\ h_k>1 \end{subarray}}^{n}
\prod\limits_{h=1}^{h_k-1}\frac{\Gamma_p\left(\left\langle\frac{p^i h}{h_k
}+\frac{p^ia}{q-1}\right\rangle \right)}{\Gamma_p\left(\left\langle \frac{hp^i}{h_k}\right\rangle\right)},
\end{align*}
where $N_{a,i}=-\sum\limits_{h=1}^{d-1} \left\lfloor\left\langle\frac{hp^i}{d}\right\rangle -\frac{ap^i}{q-1}\right\rfloor- \sum\limits_{k=1}^{n} \sum\limits_{h=0}^{h_k-1} 
\left\lfloor\left\langle\frac{-hp^i}{h_k}\right\rangle+\frac{ap^i}{q-1}\right\rfloor+\left\lfloor \frac{ap^i}{q-1}\right\rfloor.$
For $0\leq a \leq q-2$, it is observed that 
 \begin{align}\label{eq-16}
 \prod\limits_{h=1}^{h_k-1}\frac{\Gamma_p\left(\left\langle\frac{p^i h}{h_k
}+\frac{p^ia}{q-1}\right\rangle \right)}{\Gamma_p\left(\left\langle \frac{hp^i}{h_k}\right\rangle\right)}=\prod\limits_{h=1}^{h_k-1}\frac{\Gamma_p\left(\left\langle\frac{-p^i h}{h_k
}+\frac{p^ia}{q-1}\right\rangle \right)}{\Gamma_p\left(\left\langle \frac{-hp^i}{h_k}\right\rangle\right)}.
 \end{align}
Adding and subtracting the term under summation for $a=0$, and then using \eqref{eq-16} yield
\begin{align*}
	A&=(-1)^{n+1}(1-q) +(-1)^{n+1}q\sum_{a=0}^{q-2}\overline{\omega}^{a}((-1)^{d-1}\lambda^d h_1^{h_1}\cdots h_n^{h_n})(-p)^{\sum_{i=0}^{r-1}N_{a,i}}\nonumber\\
	&\hspace*{.1cm}\times\prod_{i=0}^{r-1}\Gamma_p^{n-1}\left(\left\langle\frac{p^ia}{q-1}\right\rangle \right)\prod\limits_{h=1}^{d-1}\frac{\Gamma_p
	\left(\left\langle\frac{p^ih}{d}-\frac{p^ia}{q-1}\right\rangle \right)}{\Gamma_p\left(\left\langle \frac{hp^i}{d}\right\rangle\right)}\prod\limits_{\begin{subarray}{1} \ k=1 \\ h_k>1 \end{subarray}}^{n}\prod\limits_{h=1}^{h_k-1}\frac{\Gamma_p\left(\left\langle\frac{-p^i h}{h_k
		}+\frac{p^ia}{q-1}\right\rangle \right)}{\Gamma_p\left(\left\langle \frac{-hp^i}{h_k}\right\rangle\right)}.\nonumber\\
	\end{align*}
Since $\overline{\omega}^a(-1)=(-1)^a$, we obtain 
\begin{align*}
A=(-1)^{n+1}\left(1-q-q(q-1)\cdot  {_{d-1}}G_{d-1}\left[\begin{array}{cccc}
	\frac{1}{d}, & \hspace{-.2cm}\frac{2}{d}, &\hspace{-.2cm} \ldots, & \hspace{-.2cm}\frac{d-1}{d} \\
	b_1, & \hspace{-.2cm} b_2, &\hspace{-.2cm} \ldots, & \hspace{-.2cm}b_{d-1}
	\end{array}|\lambda^d h_1^{h_1}\cdots h_n^{h_n}
	\right]_q\right),
\end{align*}
where $b_i$'s are as given in Notation \ref{notation}.
Substituting the expressions for $A$ and $B$ in \eqref{eq-17}, and then using the relation given by \eqref{eq-12.1}, we complete the proof of the theorem.
 \end{proof}
\begin{proof}[Proof of Theorem \ref{MT-2}]
	Let $N_q^d(\lambda)$ denote the number of points on the diagonal hypersurface $D_\lambda^{d,k}$ in $\mathbb{A}^{2}(\mathbb{F}_q)$. Then we have
	\begin{align}\label{eq-18}
	\# D_\lambda^{d,k} (\mathbb{F}_q)=\frac{N_q^d(\lambda) -1}{q-1}. 
	\end{align}
	Let $f(x_1,x_2)=x_1^d+x_2^d-d\lambda x_1^{k} x_2^{d-k}$ and using the identity 
	\begin{align*}
	\sum_{z\in \mathbb{F}_q} \theta(zf(x_1,x_2))= \left\{
	\begin{array}{ll}
	q, & \hbox{if $f(x_1,x_2)=0$;} \\
	0, & \hbox{if $f(x_1,x_2)\neq 0$,}
	\end{array}
	\right.
	\end{align*}
	we obtain
	\begin{align}\label{eq-19}
	q\cdot N_q^d(\lambda)&= \sum_{z\in\mathbb{F}_q}\sum_{x_i \in\mathbb{F}_q}\theta(zf(x_1,x_2))\nonumber\\
	&=q^2 + \sum_{z,x_1,x_2 \in\mathbb{F}_q^{\times}}\theta(zf(x_1,x_2)) +\sum_{z\in\mathbb{F}_q^{\times}}\sum_{\begin{subarray}{1} \ \text{some } \\ \  x_i =0\end{subarray}} \theta(zf(x_1,x_2))\nonumber\\
	&=q^2+A_1+B_1,
	\end{align}
where $A_1:=\displaystyle \sum_{z,x_i \in\mathbb{F}_q^{\times}}\theta(zf(x_1,x_2))$ and $B_1:=\displaystyle \sum_{z\in\mathbb{F}_q^{\times}}\sum_{\begin{subarray}{1} \ \text{some } \\ \  x_i =0\end{subarray}} \theta(zf(x_1,x_2))$. Firstly, we calculate $B_1$. We have
	\begin{align}\label{eq-20}
	B_1&=\sum_{z\in\mathbb{F}_q^{\times}}\sum_{\begin{subarray}{1} \ \text{some } \\ \  x_i =0\end{subarray}} \theta(zf(x_1,x_2))\nonumber\\
	&=q-1+\sum_{z\in\mathbb{F}_q^{\times}}\sum_{x_1\in\mathbb{F}_q^{\times}} \theta(zx_1^d)+\sum_{z\in\mathbb{F}_q^{\times}}\sum_{x_2\in\mathbb{F}_q^{\times}} \theta(zx_2^d).
	\end{align}
	Using Lemma \ref{lemma2_2}, we rewrite \eqref{eq-20} as
	\begin{align*}
	B_1&=q-1+\frac{2}{q-1}\sum_{z\in\mathbb{F}_q^\times}\sum_{x\in\mathbb{F}_q^{\times}}\sum\limits_{a=0}^{q-2} g(T^{-a})T^a(zx^d)\\
	&=q-1+\frac{2}{q-1}\sum\limits_{a=0}^{q-2} g(T^{-a})\sum_{z\in\mathbb{F}_q^\times}T^a(z)\sum_{x\in\mathbb{F}_q^{\times}}T^{ad}(x)\\
	&=-(q-1).
	\end{align*}
	We obtain the last equality by using the fact that inner sum is nonzero only if $a=0$ and putting $g(\varepsilon)=-1$.  Next, we simplify the expression for $A_1$. We have
	\begin{align*}
		A_1&=\sum_{z,x_1,x_2 \in\mathbb{F}_q^{\times}}\theta(zx_1^d+zx_2^d-d\lambda zx_1^kx_2^{d-k})\\
		&=\sum_{z,x_1,x_2 \in\mathbb{F}_q^{\times}}\theta(zx_1^d)\theta(zx_2^d)\theta(-d\lambda zx_1^kx_2^{d-k}).
	\end{align*}
		Applying Lemma \ref{lemma2_2}, we obtain
		\begin{align*}
			A_1&=\frac{1}{(q-1)^3}\sum_{l,m,n=0}^{q-2}g(T^{-l})g(T^{-m})g(T^{-n})\times\\
		&\hspace{1cm}\sum_{z,x_1,x_2 \in\mathbb{F}_q^{\times}}T^{l}(zx_1^d)T^{m}(zx_2^d)T^n(-d\lambda zx_1^kx_2^{d-k})\\
		&=\frac{1}{(q-1)^3}\sum_{l,m,n=0}^{q-2}g(T^{-l})g(T^{-m})g(T^{-n})T^{n}(-d\lambda)\times\\
		&\hspace{1cm}\sum_{z \in\mathbb{F}_q^{\times}}T^{l+m+n}(z)\sum_{x_1 \in\mathbb{F}_q^{\times}}T^{dl+kn}(x_1)\sum_{x_2 \in\mathbb{F}_q^{\times}}T^{dm+(d-k)n}(x_2 
		).
		\end{align*}
	The final sum is nonzero only when $n=-m-l$ by \eqref{eq-1}. Making this substitution we deduce that 
	\begin{align}\label{eq-21}
	A_1&=\frac{1}{(q-1)^2}\sum_{l,m=0}^{q-2}g(T^{-l})g(T^{-m})g(T^{l+m})T^{-l-m}(-d\lambda)\times\nonumber\\
		&\hspace{1cm}\sum_{x_1 \in\mathbb{F}_q^{\times}}T^{(d-k)l-km}(x_1)\sum_{x_2 \in\mathbb{F}_q^{\times}}T^{km-(d-k)l}(x_2).
\end{align}	 
	Inner sums in \eqref{eq-21} are nonzero only when 
	\begin{align}\label{cong}
	km\equiv (d-k)l\pmod {q-1}.
	\end{align}
	We now solve  \eqref{cong} for $l$ and $m$ and find all its possible solutions. Consider $\gcd(q-1,k)=k_0$ and $\gcd(q-1,d-k)=d_0$. Take $k_1=\frac{k}{k_0}$ and $d_1=\frac{d-k}{d_0}$.  We have $\gcd(k,d-k)=1$ as $\gcd(d,k,d-k)=1$. 
	Since $km\equiv (d-k)l\pmod {q-1}$, hence $k_0\mid l$ which implies $l=\alpha k_0$ for $0\leq \alpha\leq \frac{q-1}{k_0}-1$. Also, $\gcd(\frac{q-1}{k_0},k_1)=1$, and hence we obtain $l=k\alpha$ where $0\leq\alpha\leq\frac{q-1}{k_0}-1$. 
	Corresponding to each $l$ depending on $\alpha$, we have $k_0$ values of $m$ given by $m=(d-k)\alpha+r_1\frac{q-1}{k_0}$ where $r_1=0,1,\ldots,k_0-1$. Therefore, the total number of solutions for \eqref{cong} is $q-1$.\\
    $\emph{Claim:}$ All the solutions are given by $l\equiv ka\pmod{q-1}$ and $m\equiv(d-k)a\pmod{q-1}$ for $0\leq a \leq q-2$.\\
    \emph{Proof of the Claim}. Since $\gcd(k_0,d-k)=1$, there exists a $z_1$ such that $(d-k)z_1\equiv1\pmod {k_0}$. For any $\alpha$ and $0\leq r\leq k_0-1$, one can easily check that $l=k\alpha\equiv ka\pmod{q-1}$ and $m=(d-k)\alpha+r\frac{q-1}{k_0}\equiv(d-k)a\pmod{q-1}$ for $a=\alpha+\frac{q-1}{k_0} rz_1$.
    We now need to check if any solution is repeating or not. Suppose $ka\equiv kb\pmod{q-1}$ for some $a$ and $b$. Then $k_1a=k_1 b+\frac{q-1}{k_0} x_0$ for some $x_0$. Also, suppose that $(d-k)a\equiv(d-k)b\pmod{q-1}$. Then 
    \begin{align}\label{eq-22}
    	d_0d_1a=d_0d_1b+(q-1)y_0,
    \end{align}
    for some $y_0$. We now have $k_1d_0d_1a=k_1d_0d_1b+k_1(q-1)y_0$. Substituting the value of $k_1 a$ from above, we get $d_0d_1x_0=y_0k_0k_1$ and hence $x_0=kz$ and $y_0=(d-k)z$ for some $z$. 
    Putting the value of $y_0$ in \eqref{eq-22}, we get $(d-k)a-(d-k)b=(q-1)(d-k)z$ which gives $a\equiv b\pmod{q-1}$.
    Therefore, the solutions are given by $l\equiv ka\pmod{q-1}$ and $m\equiv(d-k)a\pmod{q-1}$ for $a=0,1,\ldots,q-2$. This completes the proof of our claim.
\par 
 Hence, the inner sums in \eqref{eq-21} are nonzero only if $l\equiv ka\pmod{q-1}$ and $m\equiv(d-k)a\pmod{q-1}$, where $a=0,1,\ldots,q-2$. Substituting these values of $l$ and $m$ in \eqref{eq-21} yields
	\begin{align*}
	A_1&=\sum_{a=0}^{q-2}g(T^{-ka})g(T^{-(d-k)a})g(T^{da})T^{-da}(-d\lambda).
	\end{align*}
Considering $n=2, \ h_1=k,$ and $h_2=d-k$ in \eqref{eq-16.1}, and proceeding along similar lines as shown in the proof of Theorem \ref{MT-1}, we deduce that 
	\begin{align*}
	A_1=q-1+q(q-1) \cdot {_{d-1}}G_{d-1}\left[\begin{array}{cccc}
		\frac{1}{d}, &\hspace{-.2cm} \frac{2}{d}, &\hspace{-.2cm} \ldots, &\hspace{-.2cm} \frac{d-1}{d}\vspace{.1cm} \\
		b_1, &\hspace{-.2cm} b_2, &\hspace{-.2cm}  \ldots, &\hspace{-.2cm} b_{d-1}
	\end{array}|\lambda^d k^{k} (d-k)^{d-k}
	\right]_q,
	\end{align*}
	where $b_i$'s are as given in Notation \ref{notation} with $h_1=k$ and $h_2=d-k$. Substituting the values of $A_1$ and $B_1$ in \eqref{eq-19} and then using the relation \eqref{eq-18}, we complete the proof of the theorem.
\end{proof}

\end{document}